\documentclass[12pt]{amsart}

\usepackage{latexsym,amssymb,amsmath}
\usepackage{tikz,verbatim,cancel}
\usepackage{multicol,fullpage}
\usepackage[pagebackref=false]{hyperref}

\numberwithin{equation}{section}
\newtheorem{theorem}{Theorem}[section]
\newtheorem{lemma}[theorem]{Lemma}

\newtheorem{corollary}[theorem]{Corollary}
\newtheorem{conjecture}[theorem]{Conjecture}

\newtheorem{question}[theorem]{Question}

\theoremstyle{definition}
\newtheorem{definition}[theorem]{Definition}
\newtheorem{construction}[theorem]{Construction}

\newtheorem{remark}[theorem]{Remark}
\newtheorem{example}[theorem]{Example}

\DeclareMathOperator{\link}{link}
\DeclareMathOperator{\del}{del}

\DeclareMathOperator{\gr}{gr}

\DeclareMathOperator{\pol}{pol}

\newcommand{\ZZ}{{\mathbb Z}}
\newcommand{\NN}{{\mathbb N}}

\def\B{{\mathcal B}}

\def\mm{{\mathfrak m}}

\def\1{{\bf 1}}
\def\0{{\bf 0}}

\begin{document}

\title[A survey of the Herzog--Hibi--Ohsugi Conjecture]{Powers of
  componentwise linear ideals: The Herzog--Hibi--Ohsugi Conjecture and related problems}

\author{Huy T\`ai H\`a}
\address{Department of Mathematics,
Tulane University, 6823 St. Charles Ave., New Orleans, LA 70118}
\email{tha@tulane.edu}
\urladdr{http://www.math.tulane.edu/~tai}

\author{Adam Van Tuyl}
\address{Department of Mathematics and Statistics,
McMaster University, Hamilton, ON, L8S 4L8, Canada}
\email{vantuyl@math.mcmaster.ca}
\urladdr{https://ms.mcmaster.ca/~vantuyl/}

\date{\today}

\begin{abstract}
  In 1999 Herzog and Hibi introduced componentwise linear
  ideals.  A homogeneous ideal $I$ is componentwise linear if for
  all non-negative integers $d$, the ideal generated by the homogeneous
  elements of degree $d$ in $I$ has a linear resolution.  For
  square-free monomial ideals, componentwise linearity is related
  via Alexander duality to the property of being sequentially Cohen-Macaulay for  the corresponding simplicial complexes.   In general,
  the property of being componentwise linear is not preserved by
  taking powers.  In 2011, Herzog, Hibi, and Ohsugi conjectured
  that if $I$ is the cover ideal of a chordal graph, then $I^s$ is
  componentwise linear for all $s \geq 1$.  We survey some
  of the basic properties of componentwise linear ideals, and then
  specialize to the progress on
  the Herzog-Hibi-Ohsugi conjecture during the last decade.  We also
  survey the related problem of determining
  when the symbolic powers of a cover ideal are componentwise linear.
  \end{abstract}

\dedicatory{Dedicated to J\"urgen Herzog on the occasion
  of his $80^{th}$ birthday}

\maketitle


\section{Introduction}

Let $R$ be a polynomial ring over a field $k$ and let $I \subseteq R$ be
a homogeneous ideal. The ideal $I$ is often considered computationally
simple if it has a \emph{linear resolution}. In 1999, Herzog and Hibi
\cite{HH1999} defined the ``next best'' class of ideals, called the
componentwise linear ideals. Specifically, by letting $\langle I_d\rangle$
be the ideal generated by homogeneous polynomials of degree $d$ in $I$,
the ideal $I$ is called a \emph{componentwise linear} ideal
if $\langle I_d\rangle$ has a linear resolution for all $d \ge 0$.

Componentwise linear ideals and their properties have become
ubiquitous in commutative algebra, especially in combinatorial
commutative algebra.  Indeed, as shown in Herzog and Hibi's
original paper, the square-free monomial ideals with
the componentwise linear property provide an algebraic
characterization of sequentially Cohen-Macaulay simplicial complexes.
Componentwise linear ideals are also closely related to another class
of ideals that have been extensively examined. A finitely generated
$R$-module $M$ is called a \emph{Koszul} module if its
associated graded module $\gr_\mm(M) = \bigoplus_{i \ge 0} \mm^iM/\mm^{i+1}M$,
with respect to the maximal homogeneous ideal $\mm$ of $R$, has a linear
resolution over the associated graded ring
$\gr_\mm(R) = \bigoplus_{i \ge 0} \mm^i/\mm^{i+1}$. Koszul algebras
(first introduced by Priddy \cite{Pri1970}) and Koszul modules (first
considered by \c{S}ega \cite{Seg2001} and by Herzog and Iyengar \cite{HI2005})
have enjoyed much attention from various areas of mathematics.
Over a polynomial ring, a homogeneous ideal $I$ is Koszul if and only
if it is componentwise linear (see \cite[Theorem 3.2.8]{R2001} and
\cite[Proposition 4.9]{Yan2000}), thus cementing the importance of
componentwise linearity.

In the two decades since their introduction, one research theme in commutative
algebra is to find new classes of componentwise linear ideals.  At the
same time, a broader research theme is to understand
how properties of an ideal are preserved when taking
powers. It is natural to ask the following question:
if $I$ is an ideal that is componentwise linear,
does the ideal $I^s$ also have this property?   The answer
to this question will be, in general, no.  In fact,
there are examples of ideals such that $I$ has a linear resolution (and
so is componentwise linear), but $I^2$ does not have a linear resolution.
For a specific example, see Example \ref{sturmfelsexample}.

One is then led to ask what extra hypotheses are required on $I$
to ensure that $I^s$ is also componentwise linear.  The first investigation
into this question was carried out by Herzog, Hibi, and Ohsugi \cite{HHO2011}.
While we will define necessary terminology in later sections, for now, it is enough to know that the ideal
$I$ in the statement below is constructed from the properties
of a finite simple graph. Based upon their results and experiments,
Herzog, Hibi and
Ohsugi posited the following conjecture:

\begin{conjecture}[Herzog-Hibi-Ohsugi] \label{conj.HHO}
  Let $I$ be the cover ideal of a chordal graph. Then $I^s$ is componentwise linear for all $s \ge 1$.
\end{conjecture}

One of the main goals
of this paper is to provide an up-to-date survey on what
is known about Conjecture \ref{conj.HHO}, and what is known about
the more general
question of powers of ideals that are componentwise linear.
Besides describing which families of chordal graphs satisfy the above
conjecture, we will also sketch out the broad strategies that have
been used to verify the conjecture.

We will also survey a variation of Conjecture \ref{conj.HHO},
which was initiated by Seyed Fakhari \cite{SF2018}.  In this
variation, it is asked whether or not $I^{(s)}$,
the $s$-th symbolic power of the cover ideal, is componentwise linear.
The two problems dovetail when the regular powers of a cover ideal
are the same as its symbolic powers.  This approach
is typified by Kumar and Kumar's recent proof that
Conjecture \ref{conj.HHO} holds for all trees (see \cite{KK2020});  in the case
of trees, regular powers and symbolic powers of cover ideals
agree, which allow Kumar and Kumar to exploit the properties
of symbolic powers of ideals.

The majority of the results in this paper have appeared
in the literature.   We have, however, included one new result
about edge ideals.  In particular we show that
the edge ideals of complete $m$-partite graphs have the
property that all of their symbolic powers
are componentwise linear (see Corollary \ref{newresult2}).  To encourage
further work on Conjecture \ref{conj.HHO}, we have included some
research questions in the final section.

We end this introduction with some final comments on our intended audience.
When writing this survey, we assumed that our readers are familiar
with minimal graded free resolutions, and possibly the
Stanley-Reisner correspondence between simplicial complexes and
square-free monomial ideals.  Our goal in Sections
2 and 3 is to provide a quick summary of some of the basic facts
surrounding componentwise linear ideals, plus
pointers to the literature for proofs of these facts.  We expect
graduate students and researchers new to componentwise
linear ideals will appreciate this approach.  In the second half of the paper,
we have been attempted to be more encyclopedic in our approach
to the Herzog-Hibi-Ohsugi Conjecture, and the related questions concerning
the symbolic powers.  We hope experts will appreciate this
snapshot of the current state-of-the-art regarding these problems.

\noindent
    {\bf Acknowledgements.}
    The authors would like to thank Aldo Conca, Nursel Erey, Ajay Kumar,
    Hop D. Nguyen,
      and Seyed Amin Seyed Fakhari for their suggestions and feedback.
    H\`a acknowledges supports from the Louisiana Board of Regents and
    the Simons Foundation.
    Van Tuyl's research is supported by NSERC Discovery Grant 2019-05412.


 \section{Basics of Componentwise Linear Ideals}

 In this section, we introduce componentwise linear ideals, and describe
 some of their basic properties.  We emphasize techniques to determine
 if an ideal is componentwise linear (a theme that will be stressed
 throughout this survey). Let $k$ be a field and let $R = k[x_1, \dots, x_n]$ denote a polynomial ring over $k$.

 We begin by recalling the notion of a minimal graded free resolution.
 Given a homogeneous ideal $I$ of $R$ (an ideal
 generated by homogeneous elements of $R$), we can associate to $I$
 a {\it minimal graded free resolution}, that is, a long exact sequence
 of length $p \leq n$ of graded $R$-modules of the form
 $$0 \rightarrow \bigoplus_{j \in \mathbb{N}}R(-j)^{\beta_{p,j}(I)}
 \rightarrow \bigoplus_{j \in \mathbb{N}}R(-j)^{\beta_{p-1,j}(I)}
 \rightarrow \cdots \rightarrow
 \bigoplus_{j \in \mathbb{N}}R(-j)^{\beta_{0,j}(I)} \rightarrow
 I \rightarrow 0.$$
 Here, $R(-j)$ denotes the twisted graded $R$-module formed by setting
 $R(-j)_d = R_{d-j}$.  The invariants $\beta_{i,j}(I)$ are
 the {\it $(i,j)$-th graded Betti numbers of $I$}, and they
 count the number of degree $j$ generators of the $i$-th
 syzygy module of $I$.

 We write $I = \langle f_1,\ldots,f_t\rangle$ if $I$ is
 generated by $\{f_1,\ldots,f_t\}$.  A homogeneous
 ideal $I = \langle f_1,\ldots,f_t \rangle$ has a
 {\it linear resolution} if
 $\deg f_1 = \cdots = \deg f_t =d$ for some integer
 $d \geq 0$ and, for all $i \geq 1$,
 $$\beta_{i,i+j}(I) = 0 ~~\mbox{for all $j \neq d$.}$$
 This is equivalent to the fact that the minimal graded free
 resolution of $I$ has the form
 $$0 \rightarrow R(-d-p)^{\beta_{p,p+d}(I)} \rightarrow \cdots \rightarrow
 R(-d-1)^{\beta_{1,1+d}(I)} \rightarrow
 R(-d)^{\beta_{0,d}(I)} \rightarrow
 I \rightarrow 0.$$

 \begin{example}\label{running}
 Let $I = \langle x_1x_2, x_2x_3,x_3x_4 \rangle$
 in $R = k[x_1,x_2,x_3,x_4]$.  Then the minimal graded free resolution
 of $I$ has the form
 $$0 \rightarrow R(-3)^2 \rightarrow R(-2)^3 \rightarrow
 I \rightarrow 0.$$
 So, the ideal $I$ has a linear resolution.

 On the other hand, consider the ideal $J = \langle x_1x_3,
 x_2x_3,x_1x_2x_4x_5 \rangle$ in $R=k[x_1,\ldots,x_5]$.
 Note that $J$ is not generated by forms of the same degree,
 so it cannot have a linear resolution.   In particular,
 the graded free resolution of $J$ has the form
 $$
 0 \rightarrow R(-3)^1 \oplus R(-5)^1 \rightarrow
 R(-2)^2 \oplus R(-4)^1 \rightarrow
 J \rightarrow 0.$$
 We will return to these ideals below.
 \end{example}

 We are now in a position to define the main objects of
 this survey.  For any homogeneous ideal $I \subseteq R$
 and integer $d \geq 0$, let
 $$\langle I_d \rangle = \langle \{ F\in I ~|~ \mbox{$F$ is homogeneous
 of degree $d$}\} \rangle$$
 denote the ideal generated by all the homogeneous elements
 of degree $d$ in $I$.  The following definition is then due
 to Herzog and Hibi \cite{HH1999}:

 \begin{definition}[Herzog--Hibi]
   A homogeneous ideal $I$ of $R$ is {\it componentwise linear}
   if for all integers $d \geq 0$, the ideal $\langle I_d \rangle$
   has a linear resolution.
 \end{definition}

 Based only upon the above definition,
 verifying whether or not a homogeneous
 ideal $I$ is componentwise linear would involve checking
 an infinite number of conditions.  Fortunately, it
 is possible to verify if $I$ is componentwise linear by only checking
 a finite number of $d \in \mathbb{N}$ by using the
 Castelnuovo-Mumford regularity of $I$.

 The {\it (Castelnuovo-Mumford) regularity of $I$} is defined to be
 $${\rm reg}(I) = \max\{j-i ~|~ \beta_{i,j}(I) \neq 0 \}.$$
 Roughly speaking, ${\rm reg}(I)$ measures the largest degree
 of a generator of a syzygy of $I$.  This invariant can
 also be viewed as a measure of the complexity of $I$.  We then
 require the following well-known fact (see, for example,
 \cite[Proposition 1.1]{EG}).

 \begin{theorem}\label{regularitybound}
   Let $I$ be a homogeneous ideal with $r = {\rm reg}(I)$.
   Then for all $d \geq r$, the ideal $\langle I_d \rangle$ has
   a linear resolution.
 \end{theorem}

 Consequently, determining whether $I$ is componentwise
 linear can be reduced to checking a finite number of cases.

 \begin{corollary}\label{regcwl}
   Let $I$ be a homogeneous ideal with
   $r = {\rm reg}(I)$. Then $I$ is componentwise
   linear if and only if $\langle I_d \rangle$ has a linear
   resolution for all $0 \leq d
   \leq r$. In particular, if $I$ has a linear resolution,
   then $I$ is also componentwise linear.
 \end{corollary}

 \begin{proof}
   The if-and-only if statement follows directly from Theorem
   \ref{regularitybound} and the definition.
   If $I$ has a linear resolution, then all the
   homogeneous generators of $I$ have the same degree, which coincides with its regularity $r$.
   Thus $\langle I_d \rangle = \langle 0 \rangle$ for all
   $0 \leq d < r$, and $\langle I_r \rangle = I$.  The conclusion now follows from the first
   part of the statement.
 \end{proof}

 \begin{example} Let
   $J = \langle x_1x_3, x_2x_3,x_1x_2x_4x_5 \rangle$ in $R=k[x_1,\ldots,x_5]$
   be the ideal of Example \ref{running}.  From the minimal
   graded free resolution of $J$ given in Example \ref{running}, we have
   $${\rm reg}(J) = \max\{3-1,5-1,2-0,4-0\} = 4.$$
   So, by Corollary \ref{regcwl}, we need to compute the
   minimal graded free resolutions of $\langle J_d \rangle$ for $d=0,\ldots,4$.
   Since $\langle J_d \rangle = \langle 0 \rangle$ for $d=0,1$, we
   focus on the remaining cases.  For the ideal
   $\langle J_2 \rangle = \langle x_1x_3,x_2x_3 \rangle$ we have
   $$ 0 \rightarrow R(-3) \rightarrow R(-2)^2 \rightarrow \langle J_2 \rangle
   \rightarrow 0;$$
   for the ideal
   $$\langle J_3 \rangle = \langle x_1^2x_3,x_1x_2x_3,x_1x_3x_4,x_1x_3x_5,
   x_2^2x_3,x_2x_3^2,x_2x_3x_4,x_2x_3x_5 \rangle$$
   the resolution is
   $$0 \rightarrow R(-7)^2 \rightarrow R(-6)^{10} \rightarrow R(-5)^{20}
   \rightarrow R(-4)^{20} \rightarrow R(-3)^9 \rightarrow \langle J_3 \rangle
   \rightarrow 0; $$
   and for the ideal $\langle J_4 \rangle$ (we suppress the 26 minimal generators)
   the resolution has the form
   $$0 \rightarrow R(-8)^9 \rightarrow R(-7)^{43} \rightarrow R(-6)^{80}
   \rightarrow R(-5)^{71} \rightarrow R(-4)^{26} \rightarrow \langle J_4 \rangle
   \rightarrow 0.$$
   Thus, the ideal $J$ is componentwise linear.
 \end{example}

 We give some alternative ways to characterize when an ideal
 is componentwise linear.  One such characterization is
 in terms of the
 generic initial ideal.  We review the relevant terminology; for more
 on generic initial ideals see \cite{G1998}.

 Let $GL_n(k)$ denote the general linear group of order $n$
 over $k$, i.e.,  all the $n\times n$ invertible matrices with
 entries in $k$.  Any matrix
 $g \in GL_n(k)$ acts on the variables $x = \{x_1, \ldots, x_n\}$
 by a linear change of variables, i.e., $x_i$ is sent to
 $g_{i,1}x_1+g_{i,2}x_2 + \cdots + g_{i,n}x_n$ for $i=1,\ldots,n$,
 where $(g_{i,1},\ldots,g_{i,n})$ is the $i$-th row of $g$.
Given $g \in GL_n(k)$ and a polynomial $f = f(x_1,\ldots,x_n) \in R$,
then $g$ acts on $f$ by
$g\cdot f := f(g \cdot x).$
Now fix an ideal $I$ of $R$ and monomial order $>$.
Every matrix $g \in GL_n(k)$ results in an initial ideal
$\operatorname{in}_>(g\cdot I)$ where
$g\cdot I = \langle g\cdot f ~|~ f \in I \rangle.$
We say that two
matrices $g$ and $g'$ are \emph{equivalent} if
$$\operatorname{in}_>(g\cdot I) = \operatorname{in}_>(g'\cdot I).$$
This definition
induces an equivalence relation on $GL_n(k)$, and thus,
the equivalence classes partition the
group $GL_n(k)$.

One of these partitions is quite ``large'' in
the following sense:

\begin{lemma}[{\cite[Theorem 4.1.2]{HH2011}}]\label{Zariski}
For a fixed $I$ and term order $>$, one of the equivalence
classes is a nonempty Zariski open subset $U$  inside $GL_n(k)$.
\end{lemma}

Although we do not go into the details here, the Zariski open set inside of
$GL_n(k)$ refers to the zero set of some ideal in $k[y_{i,j}~|~
1 \leq i,j\leq n]$.   Note that for all $g \in U$, the Zariski open subset in Lemma \ref{Zariski},
the initial ideal $\operatorname{in}_>(g\cdot I)$ is the same ideal.

\begin{definition}
Fix a term order $>$ on $R$, let $I$ be an ideal of $R$, and let
$g$ be any element of the open Zariski subset of Lemma \ref{Zariski}.
The initial ideal $\operatorname{in}_>(g\cdot I)$ is called the
\emph{generic initial ideal} of $I$ for the
term order $>$.  It is
denoted $\operatorname{gin}_>(I) = \operatorname{in}_>(g\cdot I)$.
\end{definition}

Roughly speaking, the generic initial ideal is
the ideal we should expect if we pick a ``random'' matrix
$g \in GL_n(k)$ and form $\operatorname{in}_>(g\cdot I).$

With this terminology, we then have the following equivalent
statements. In particular, an ideal $I$ is componentwise linear
if $I$ and the generic initial ideal $I$ with respect to
the reverse lexicographical order have the same number of generators in each degree.  We
want to highlight that this statement also requires the
hypothesis that the field $k$ has characteristic zero.

\begin{theorem}\label{equivalences}
  Let $I$ be a homogeneous ideal of $R$.
  Assume ${\rm char}(k) = 0$.  Then the
  following are equivalent:
  \begin{enumerate}
  \item $I$ is componentwise linear
  \item $\beta_{i,j}(I) = \beta_{i,j}({\rm gin}_>(I))$ for all $i,j\geq 0$
    where $>$ is the reverse lexicographical order.
  \item $\beta_{0,j}(I) = \beta_{0,j}({\rm gin}_>(I))$ for all $j \geq 0$
    where $>$ is the reverse lexicographical order.
  \end{enumerate}
    \end{theorem}

\begin{proof}
  The equivalence of (1) and (2) is \cite[Theorem 1.1]{AHH2000}.
  The equivalence of (1) and (3) first appears in the
  paper of Conca \cite[Theorem 1.2]{C2004}, although Conca points out
  that this equivalence is implicit
  in the proof of \cite[Theorem 1.1]{AHH2000}.
  \end{proof}

\begin{example}
  Theorem \ref{equivalences} gives an alternative way to determine if
  an ideal is componentwise linear.  The computer algebra
  package Macaulay2 \cite{M2} is able to compute the generic
  initial ideal (although the algorithm is only probabilistic in the sense
  that it computes the initial ideal $g\cdot I$ for a random $g \in GL_n(k)$).

  Here is a sample session applied to the componentwise linear ideal $J$
  of Example \ref{running}.

  \footnotesize
\begin{verbatim}
i1 : R = QQ[x_1..x_5]

i2 : j = monomialIdeal(x_1*x_3,x_2*x_3,x_1*x_2*x_4*x_5)

i3 : loadPackage "GenericInitialIdeal"

i4 : gin j

o4 = ideal (x_1^2, x_1*x_2, x_2^4)
\end{verbatim}

\normalsize
Since the ideal and its generic initial ideal have the same
number of generators in each degree, the ideal is componentwise linear.
\end{example}

As mentioned in the introduction, componentwise linear ideals
are related to the notion of Koszulness.  We record this equivalence.

\begin{theorem}\label{koszul}
  Let $I$ be a homogeneous ideal of $R$.
  Assume ${\rm char}(k) = 0$.  Then
  $I$ is componentwise linear if and only if the $R$-module $I$ is a Koszul
  module.
\end{theorem}

\begin{proof}
  See R\"omer
  \cite[Theorem 3.2.8]{R2001}
  or Yanagawa \cite[Proposition 4.9]{Yan2000}.
\end{proof}

\begin{remark}
  Because of the equivalence of Theorem \ref{koszul},
  componentwise linear ideals are sometimes called
  {\it Koszul ideals}; for example, see \cite{DHNT2021}.
  \end{remark}

We end this section by introducing linear quotients,
an extremely useful technique to show that an ideal is componentwise linear.
Ideals with linear quotients were first defined by
Herzog and Takayama \cite{HT2002} for monomial ideals;  the
more general definition is given below.

\begin{definition}
  A homogeneous ideal $I$ has {\it linear quotients} if the  minimal
  generators of $I$ can be ordered as $f_1,\ldots,f_s$ such that for
  each $i=2,\ldots,s$, the ideal
  $\langle f_1,\ldots,f_{i-1} \rangle:\langle f_i \rangle$ is generated by
  linear forms.
\end{definition}

Linear quotients can then be used to verify that an ideal
is componentwise linear:

\begin{theorem}[{\cite[Theorem 8.2.15]{HH2011}}]
  \label{thm.LQimpliesCWL}
  Let $I$ be a homogeneous ideal, and suppose that $I$ has
  linear quotients with respect to a minimal set of generators of $I$.
  Then $I$ is componentwise linear.
\end{theorem}

\begin{example} If we consider the ideal $J = \langle x_1x_3,x_2x_3,x_1x_2x_4x_5 \rangle$ of Example \ref{running}, we have
$$
    \langle x_1x_3 \rangle:\langle x_2x_3 \rangle  =  \langle x_1 \rangle ~~~\mbox{and}~~~
    \langle x_1x_3, x_2x_3 \rangle:\langle x_1x_2x_4x_5\rangle =\langle x_1,x_2 \rangle.
$$
  The ideal $J$ has linear quotients, thus giving another way
  of seeing that $J$ is componentwise linear.
\end{example}

\begin{remark}\label{reisnerex}
  We want to stress that  the property of being
  componentwise linear (or having a linear resolution) depends
  upon the characteristic of the field $k$.   A well known example
  of this phenomenon, due to Reisner \cite{R1976}, is
  the square-free monomial ideal
  $$I =
\langle x_1x_2x_3,x_1x_2x_6,x_1x_3x_5,x_1x_4x_5,x_1x_4x_6,x_2x_3x_4,
 x_2x_4x_5,x_2x_5x_6,x_3x_4x_6,
x_3x_5x_6 \rangle$$
in $R =k[x_1,\ldots,x_6]$. The ideal $I$ has a linear resolution if and
only if ${\rm char}(k) \neq 2$.  So $I$ is componentwise linear if and
only if ${\rm char}(k) \neq 2$.
This example also shows that the converse of Theorem \ref{thm.LQimpliesCWL}
cannot hold, since linear quotients is a
property that does not ``see'' the characteristic of the field.
\end{remark}


\section{Componentwise linearity of (square-free) monomial ideals}

In this section we recall properties of (square-free) monomial
ideals that are componentwise linear.  A highlight of this section
is Herzog and Hibi's classification of square-free monomial ideals
that are componentwise linear, which was one of main results in
the paper that introduced componentwise linearity \cite{HH1999}.

Let $V = \{x_1,\ldots,x_n\}$ be a collection of vertices.  A {\it
  simplicial complex} $\Delta$ on $V$ is a subset of the power set of
$V$ that satisfies the following two properties:  (1) if $F \in \Delta$
and if $G \subseteq F$, then $G \in \Delta$; and (2) $\{x_i\} \in \Delta$
for all $i=1,\ldots,n$.  The maximal elements of $\Delta$ ordered with
respect to inclusion are called the {\it facets} of $\Delta$.  If
$F_1,\ldots,F_s$ is a complete list of the facets of $\Delta$, then
we usually write $\Delta = \langle F_1,\ldots,F_s \rangle$. In this
case, we say $\Delta$ is generated by $F_1,\ldots,F_s$.  An
element $F \in \Delta$ is called a {\it face} of $\Delta$.  The
{\it dimension} of $F$ is $\dim F = |F|-1$ (we use
the convention that $\dim \emptyset = -1$).  The {\it dimension}
of a simplicial complex is $\dim \Delta = \max \{ \dim F ~|~ F \in \Delta\}$.
A simplicial complex is {\it pure} if all of its facets have
the same dimension.

We associate with $\Delta$ a square-free monomial ideal
$I_\Delta$ in the polynomial ring $R = k[x_1,\ldots,x_n]$ as follows:
$$I_\Delta = \langle x_{i_1}\cdots x_{i_s} ~|~ \{x_{i_1},\ldots,x_{i_s}\} \not\in
\Delta \rangle.$$
The ideal $I_\Delta$ is the {\it Stanley-Reisner ideal} of $\Delta$; it
captures many of the combinatorial properties of $\Delta$.
This construction  can be reversed, that is, given any square-free monomial
ideal $I$, we can construct its \emph{Stanley-Reisner simplicial complex}
$$\Delta_I = \{ \{x_{i_1},\ldots,x_{i_s}\} ~|~
\mbox{$x_{i_1}\cdots x_{i_s}$ is a square-free monomial not in $I$} \}.$$

\begin{example}
  Consider our running example (see Example \ref{running})
  $$J = \langle x_1x_3,x_2x_3,x_1x_2x_4x_5\rangle \subseteq
  k[x_1,\ldots,x_5].$$
  On the vertex set $V = \{x_1,\ldots,x_5\}$, we then have
  $$\Delta_J = \langle
  \{x_3,x_4,x_5\},\{x_2,x_4,x_5\},\{x_1,x_4,x_5\},\{x_1,x_2,x_5\}, \{x_1,x_2,x_4\} \rangle.$$
\end{example}

We say that $\Delta$ is a {\it (sequentially) Cohen-Macaulay} simplicial complex
if the quotient ring $R/I_\Delta$ is (sequentially) Cohen-Macaulay. Particularly, sequentially Cohen-Macaulay simplicial complexes are described as follows:

\begin{definition}
  Let $\Delta$ be a simplicial complex with $\dim \Delta = d$.  For each
  $i=-1,\ldots,d$, let $\Delta(i) = \langle F \in \Delta ~|~
  \dim F = i \rangle,$ i.e., the simplicial complex generated
  by all the faces of dimension $i$ in $\Delta$.
  Then $\Delta$ is {\it sequentially Cohen-Macaulay} if $\Delta(i)$
  is Cohen-Macaulay for all $i=-1,\ldots,d$.
  \end{definition}

Given a simplicial complex $\Delta$, the {\it Alexander dual} of $\Delta$
is the simplicial complex
$$\Delta^\vee = \langle V\setminus F ~|~ F \not\in \Delta \rangle.$$
Using the above terminology, square-free
monomial ideals that are componentwise linear can be classified.
This result, which is due to Herzog and Hibi, is one of the first
major results about componentwise linear ideals.

\begin{theorem}[{\cite[Theorem 2.1]{HH1999}}]\label{scmstatement} Let $I$ be
  a square-free monomial ideal and let $\Delta = \Delta_I$ be its Stanley-Reisner simplicial complex.  Then $I$ is componentwise linear if and only if $\Delta^\vee$ is sequentially Cohen-Macaulay.
\end{theorem}

\begin{example}
  As noted multiple times, the ideal
  $J = \langle x_1x_3,x_2x_3,x_1x_2x_4x_5\rangle $ is componentwise linear.
  So the simplicial complex
  $$\Delta^\vee_J = \langle \{x_3\},\{x_2,x_4,x_5\},\{x_1,x_4,x_5\}\rangle$$
  is sequentially Cohen-Macaulay.
  \end{example}

The above theorem generalizes an important result of Eagon and Reiner.
We also record this result.

\begin{theorem}[{\cite[Theorem 3]{ER1998}}]\label{eagonreiner}
  Let $I$ be a square-free monomial ideal and let $\Delta=\Delta_I$ be its Stanley-Reisner simplicial complex.  Then $I$ has a linear resolution
  if and only if $\Delta^\vee$ is Cohen-Macaulay.
  \end{theorem}

\begin{example}
  If $I$ is the square-free monomial ideal of Remark \ref{reisnerex},
  then the ideal $I$ has a linear resolution if and
only if ${\rm char}(k) \neq 2$.  So $\Delta^\vee_I$ is Cohen-Macaulay
if and only if ${\rm char}(k) \neq 2$.
  \end{example}

Theorem \ref{scmstatement} provides a new strategy to prove that
an ideal is componentwise linear.  In particular, instead of showing
that $I$ is componentwise linear, it is enough to show that $\Delta^\vee_I$
is sequentially Cohen-Macaulay.  There are two combinatorial
ways to determine if a simplicial complex is sequentially Cohen-Macaulay;
we first give some relevant definitions.

For a vertex $x$ in $\Delta$, the \emph{deletion} of $x$ in $\Delta$, denoted by $\del_\Delta(x)$, is the simplicial complex obtained by removing $x$ and all faces containing $x$ from $\Delta$. Also, the \emph{link} of $x$ in $\Delta$, denoted by $\link_\Delta(x)$, is the simplicial complex whose faces are
	$$\{F \in \Delta ~\big|~ x \not\in F \text{ and } F \cup \{x\} \in \Delta\}.$$

\begin{definition} Let $\Delta = \langle F_1,\ldots,F_s \rangle$
  be a simplicial complex.
  \begin{enumerate}
  	\item The complex $\Delta$ is {\it shellable} if there exists a linear order of its facets $F_1, \dots, F_s$ such that for all $i = 2, \dots, s$, the subcomplex $\langle F_1, \dots, F_{i-1}\rangle \cap \langle F_i\rangle$ is pure and of dimension $(\dim F_i -1)$.
\item The complex $\Delta$ is {\it vertex decomposable} if either:
\begin{enumerate}
\item $\Delta$ is a {\it simplex} (i.e., it has a unique facet);
  or the empty complex; or
	\item there exists a vertex $x$ in $\Delta$ such that all facets of $\del_\Delta(x)$ are facets of $\Delta$ (i.e.,$x$ is a {\it shedding vertex}), and both $\link_\Delta(x)$ and $\del_\Delta(x)$ are vertex decomposable.
\end{enumerate}
    \end{enumerate}
  \end{definition}

\begin{remark}
The notions of shellability and vertex decomposability were first given for {\it pure} simplicial complexes, that is, when all the facets
have the same dimension.  These notions were generalized by
Bj\"orner and Wachs \cite{BW1996,BW1997} to include nonpure simplicial complexes.
\end{remark}

One approach in the literature (which we will return to later in the paper)
to show a square-free monomial ideal is componentwise linear is to show that
the associated simplicial complex has one of the above properties.

\begin{theorem}\label{vd=>shellable=>scm}
  Let $I$ be a square-free monomial ideal with
  associated simplicial complex $\Delta = \Delta_I$.
  If $\Delta^\vee$, the Alexander dual of $\Delta$, is either
  vertex decomposable or shellable, then $I$ is componentwise linear.
\end{theorem}

\begin{proof}
  By \cite[Theorem 11.3]{BW1997}, a vertex decomposable
  simplicial complex $\Gamma$ is also shellable.
  Stanley (see page 87 of \cite{S1996}) first observed
  that if $\Gamma$ is a shellable simplicial complex, then
  $\Gamma$ is sequentially Cohen-Macaulay.  The conclusion
  then follows from Theorem \ref{scmstatement} since both
  hypotheses imply that $\Delta^\vee$ is sequentially Cohen-Macaulay.
\end{proof}

For the case of square-free monomial ideals, there is an alternative
way to verify that the ideal is componentwise linear.  Given
a square-free monomial ideal $I$, let $I_{[d]}$ be the ideal generated
by all the square-free monomials of degree $d$ in $I$.  We then say $I$ is
{\it square-free componentwise linear} if $I_{[d]}$ has a linear resolution
for all $d$.  Checking whether a square-free monomial ideal is componentwise
linear then reduces to checking whether or not it is square-free
componentwise linear.

\begin{theorem}[{\cite[Proposition 1.5]{HH1999}}]\label{sqfreecwl}
   Let $I$ be a square-free
    monomial ideal.  Then $I$ is componentwise linear if and only if
    $I$ is square-free componentwise linear.
\end{theorem}

\begin{remark}
  If $I$ is a square-free monomial ideal in $R = k[x_1,\ldots,x_n]$, then $I$ cannot have any square-free monomials of degree $> n$.  So,
  to check if $I$ is square-free componentwise linear, we only need
  to check that $I_{[d]}$ has a linear resolution for $0 \leq d \leq n$.  In fact,
  since there is only one square-free monomial of degree $n$,
  namely $x_1\cdots x_n$, either $I_{[n]} = \langle 0 \rangle$ or
  $I_{[n]}= \langle x_1\cdots x_n \rangle$, and so $I_{[n]}$ always has a linear
  resolution.  So we only need to check for $0 \leq d < n$.
  \end{remark}

We expand our scope to now say a few words about monomial ideals
more generally, and not just the square-free case.  One common
approach to studying monomial ideals is to use the process
of polarization to turn the monomial ideal into a square-free monomial
ideal in a much larger polynomial ring.  In many instances,
properties of the original ideal are preserved in the new ideal, and
vice versa.   It turns out that the linear quotient property
is preserved under this operation.

We formally define the polarization procedure.
Let $I = \langle x_1^{a_{1,1}}\cdots x_n^{a_{1,n}},\ldots,
x_1^{a_{t,1}}\cdots x_n^{a_{t,n}} \rangle$ be a monomial ideal
in $R = k[x_1,\ldots,x_n]$.  For $j=1,\ldots,n$, set
$b_j = \max\{a_{i,j} ~|~ 1 \leq i \leq t \}$, that is, $b_j$ is
the highest power of $x_j$ that appears among the generators of
$I$.  In a polynomial ring $$S=k[x_{1,1},\ldots,x_{1,b_1},x_{2,1},\ldots,x_{2,b_2},
  \ldots,x_{n,1},\ldots,x_{n,b_n}]$$
we define the {\it polarization} of $I$ to be the ideal
\footnotesize
$$
  I^{{\rm pol}} = \langle x_{1,1}\cdots x_{1,a_{1,1}}x_{2,1}\cdots x_{2,a_{1,2}}
  \cdots x_{n,1}\cdots x_{n,a_{1,n}}, \ldots,
  x_{1,1}\cdots x_{1,a_{t,1}}x_{2,1}\cdots x_{2,a_{t,2}}
  \cdots x_{n,1}\cdots x_{n,a_{t,n}} \rangle.
  $$
    \normalsize
    That is, we replace $x_j^{a_{k,j}}$ with $x_{j,1}\cdots x_{j,a_{k,j}}$
    in each generator of $I$.

    As we saw earlier, linear quotients is a technique that can
    be used to check if an ideal is componentwise linear.  The following
    result of Seyed Fakhari shows that for an arbitrary monomial ideal,
    we can use the polarization of the ideal to check if the original
    ideal has linear quotients.

\begin{theorem}[{\cite[Lemma 3.5]{SF2018}}]\label{linquot-polar}
  Let $I$ be a monomial ideal.
  Then $I$ has linear quotients if and only $I^{\rm pol}$, the polarization of $I$,
  has linear quotients.
  \end{theorem}

\begin{example}
  Let $I = \langle y_1y_3,y_2y_3,y_1y_2y_4^2 \rangle$ in $R=k[y_1,y_2,y_3,y_4]$.
  We then
  have $b_1 = 1,~b_2 = 1,~b_3=1$, and $b_4=2$.  In the polynomial ring
  $S = k[y_{1,1},y_{2,1},y_{3,1},y_{4,1},y_{4,2}]$, the polarization
  of $I$ is $$I^{\rm pol} = \langle y_{1,1}y_{3,1},y_{2,1}y_{3,1},y_{1,1}y_{2,1}y_{4,1}y_{4,2}\rangle.$$
  If we relabel the variables
  so that $x_1 = y_{1,1}, x_2 = y_{2,1}, x_3 = y_{3,1}, x_4 = y_{4,1}$
  and $x_5 =y_{4,2}$, then $I^{\rm pol}$ is the ideal $J$ of our running example.
  So $I$ has linear quotients (which can also be checked directly).
\end{example}

There are many classes of monomial ideals which have been identified
as componentwise linear.  While we do not survey all of this literature
(since we wish to focus on powers of ideals), we highlight some
families that are relevant for our future discussions.

A monomial ideal $I$ is {\it weakly polymatroidal} if for every
pair of minimal generators $m_1 = x_1^{a_1}\cdots x_n^{a_n}$ and
$m_2 = x_1^{b_1}\cdots x_{n}^{b_n}$ with $m_1 > m_2$
with respect to the the lexicographical
ordering, and if $a_1=b_1,\ldots,a_{t-1}=b_{t-1}$ but $a_t > b_t$,
then there exists a $j > t$ such that $x_t(m_2/x_j) \in I$.
The notion of weakly polymatroidal generalizes the notation of a stable
ideal.  An ideal $I$ is {\it stable} if for any monomial
$m =  x_{i_1}^{a_{i_1}}\cdots x_{i_s}^{a_{i_s}}$ in $I$, if $j < {i_s}$,
then $x_j(m/x_{i_s}) \in I$.   As shown by Mohammadi and Moradi,
these ideals all have linear quotients, thus providing us
with a large class of componentwise linear monomial ideals.

\begin{theorem}[{\cite[Theorem 1.3]{MM2010}}]\label{thm.weaklypoly}
  If $I$ is a weakly polymatroidal ideal, then
  $I$ has linear quotients, and consequently, $I$ is componentwise
  linear.
\end{theorem}

We now introduce another class of ideals.
For any $J = \{j_1,\ldots,j_s\} \subseteq \{1,\ldots,n\}$, we define
$\mathfrak{m}_J = \langle x_{j_1},\ldots,x_{j_s} \rangle$.  For
any integer $a \geq 1$, we call $\mathfrak{m}_J^a$ a {\it Veronese ideal}
(see \cite{HH2006}).
Given subsets $J_1,\ldots,J_s$ of $\{1,\ldots,n\}$ and
positive integers $a_1,\ldots,a_s$ we call
$$I = \mathfrak{m}_{J_1}^{a_1} \cap \cdots \cap \mathfrak{m}_{J_s}^{a_s}$$
an {\it intersection of Veronese ideals}.  Ideals of this type
appear throughout the literature (for example, the primary decomposition
of a square-free monomial ideal can viewed as the intersection of
Veronese ideals).  In some cases, we can determine
if $I$ is componentwise linear simply from the subsets $J_1,\ldots,J_s$.
One such example is the following result of
Francisco and Van Tuyl \cite[Theorem 3.1]{FVT2007b}, and generalized
by Mohammadi and Moradi (whose result is presented below).

\begin{theorem}[{\cite[Theorem 2.5]{MM2010}}]\label{veronese}
  Let $J_1,\ldots,J_s,K$ be subsets of $[n] = \{1,\ldots,n\}$.
  Suppose that $J_i \cup J_j = [n]$ for
  all $i \neq j$ and
  $K \subseteq [n]$.  Then
  $$I = \mathfrak{m}_{J_1}^{a_1} \cap \cdots \cap \mathfrak{m}_{J_s}^{a_s}
  \cap \mathfrak{m}_K^b$$
  is componentwise linear for any positive integers $a_1,\ldots,a_s,b$.
  \end{theorem}

We end this section with a recent result of Dung, Hien, Nguyen, and Trung
that allows one to build new componentwise linear ideals from
old ones.  We have only presented the monomial ideal version of this result,
although the work of \cite{DHNT2021} is more general since the
focus of their work is the linear defect of an ideal.

\begin{theorem}[{\cite[Corollary 5.6]{DHNT2021}}]\label{buidcwl}
  Let $R = k[x_1,\ldots,x_n]$.  Let $I'$ and $T$ be non-trivial
  monomial ideals and $x$ a variable such that
  \begin{enumerate}
  \item $I'$ is componentwise linear,
  \item $T \subseteq \langle x_1,\ldots,x_n\rangle I'$, and
  \item no generator of $T$ is divisible by $x$.
  \end{enumerate}
  If $I = xI'+T$, then $I$ is componentwise linear if and only if
  $T$ is componentwise linear.
\end{theorem}


\section{Componentwise linearity of regular powers}

In this section we survey the problem of determining when the
regular powers of an ideal is  componentwise linear.  In particular,
we will focus on the Herzog-Hibi-Ohsugi Conjecture on the behavior
of cover ideals of chordal graphs.

One theme in commutative algebra is to understand how properties
of an ideal are preserved when one takes powers of these ideals.
This theme is encapsulated into the following broad question:

\begin{question}\label{maintheme}
  Let $I$ be an ideal of a ring $T$.  Suppose that the ideal
  $I$ has some property $\mathcal{P}$.  Does $I^s$ also have
  property $\mathcal{P}$ for all integers $s \geq 1$?
  \end{question}
\noindent
The monograph \cite{CHHVT2020} looks at this question
for a number of ideals that arise in either combinatorics or geometry.
Given this theme, it is natural to ask if the property of
being componentwise linear is preserved by taking powers.
The answer turns out to be no in general as shown in the follow examples.

\begin{example}\label{sturmfelsexample}
    Let
    $$I =
\langle x_1x_2x_3,x_1x_2x_6,x_1x_3x_5,x_1x_4x_5,x_1x_4x_6,x_2x_3x_4,
 x_2x_4x_5,x_2x_5x_6,x_3x_4x_6,
x_3x_5x_6 \rangle$$
in $R =k[x_1,\ldots,x_6]$
be Reisner's example as given in Example \ref{reisnerex}.  As already
noted, this ideal has a linear resolution if ${\rm char}(k) \neq 2$.
It was shown by Conca in \cite[Remark 3]{C2000} that
when ${\rm char}(k) \neq 2$, the ideal
$I^2$ does not have a linear resolution.  Since
$I^2 = \langle (I^2)_6 \rangle$, the square of $I$ is not
componentwise linear. Note that in \cite{C2000}, this example was attributed to Terai.

Sturmfels \cite{S2000} gave another example, which
    does not depend upon the characteristic of the ground field.  In particular,
    the ideal
    $$I = \langle
    x_1x_3x_6,x_1x_4x_5,x_2x_3x_4,x_2x_5x_6,x_3x_4x_5,x_3x_4x_6,x_3x_5x_6,x_4x_5x_6
    \rangle$$
has the property that $I$ has a linear resolution in all
characteristics, but $I^2$ does not have a linear resolution.

Conca's paper \cite{C2006} also contains many other examples of
(monomial) ideals with this behavior.
\end{example}

The previous example shows that we will require some extra hypotheses
on $I$ in order to guarantee that $I^s$ is componentwise linear.
Restricting to the case that $I$ has a linear resolution is a natural
starting point, and it turns out that we can say more in this case.

We say that a homogeneous ideal $I$ has {\it linear powers}
(following Bruns, Conca, and Varbaro \cite{BCV2015})
if $I$ has a linear resolution and
$I^s$ has a linear resolution for all $s \geq 2$.
Ideals with linear powers can be classified in terms of their
Rees algebras.  If $I = \langle f_1,\ldots,f_s \rangle$ is generated
by homogeneous elements of degree $d$, the {\it Rees algebra} is
$${\rm Rees}(I) = \bigoplus_{s \in \mathbb{N}} I^s.$$
This ring has a bigraded structure given by ${\rm Rees}(I)_{a,b} = (I^b)_a$
with $(a,b) \in \mathbb{N}^2$, that
is, all the elements of degree $a$ in the $b$-th power of $I$.
We can give ${\rm Rees}(I)$ a graded structure by setting the degree $a$
part of ${\rm Rees}(I)$ to be
$${\rm Rees}(I)_{(a,*)} = \bigoplus_{b \in \mathbb{N}}(I^b)_a.$$
With this grading ${\rm Rees}(I)$ is a graded $R$-module.  We denote
its Castelnuovo-Mumford regularity with respect to this
grading by ${\rm reg}_{(1,0)} {\rm Rees}(I)$.   (There are other possible
$\mathbb{N}$-gradings one can put on ${\rm Rees}(I)$, so one wants
to distinguish the grading used when taking the regularity.) We then have:

\begin{theorem}[{\cite[Theorem 2.5]{BCV2015}}]\label{linearpowers}
  Let $I = \langle f_1,\ldots,f_s \rangle$ be a homogeneous
  ideal generated by forms of the same degree.  Then
  $I$ has linear powers if and only if ${\rm reg}_{(1,0)}{\rm Rees}(I) =0.$
\end{theorem}

\noindent
This result is generalized to the case of $I^sM$ for a module $M$ in
\cite{BCV2021}.   The ``if'' direction was first proved
by R\"omer \cite[Corollary 5.5]{R2001}.

While Example \ref{sturmfelsexample} shows that
  we should not expect arbitrary products of componentwise linear ideals
  to be componentwise linear, Conca, De Negri, and Rossi \cite{CDR2010}
  gave a sufficient condition for this property.

\begin{theorem}[{\cite[Theorem 2.20]{CDR2010}}]
  Suppose that $I$ and $J$ are componentwise linear, and suppose $d$ is
  the smallest degree of a generator of $I$.  If $\dim R/\langle I_d \rangle \leq 1$, then $IJ$ is componentwise linear. In particular, if
  $\dim R/\langle I_d \rangle \leq 1$, then $I^s$ is componentwise
 linear for all $s \geq 1$.
\end{theorem}

Another way to approach Question \ref{maintheme} is to
restrict to families of ideals that are known to be componentwise
linear, and check if the powers of ideals within this family continue to
be componentwise linear.  One such family of ideals is the cover ideals
of graphs.  We recall the relevant definitions and notation.

Let $G = (V,E)$ be a finite simple graph on the vertex set
$V = \{x_1,\ldots,x_n\}$ and edge set $E$, which consists of unordered
pairs of distinct elements of $V$.  By identifying the vertex $x_i \in V$
with the variable in $x_i$ in $R = k[x_1,\ldots,x_n]$, we
can associate to $G$
two square-free monomials ideals, the {\it edge ideal}
$$I(G) = \langle x_ix_j ~|~ \{x_i,x_j\} \in E \rangle,$$
and the {\it cover ideal}
$$J(G) = \bigcap_{\{x_i,x_j\} \in E} \langle x_i,x_j \rangle.$$
The terminology of edge ideal is used to highlight the fact that
the minimal generators of $I(G)$ correspond to the edges of the graph.
For the cover ideal, the minimal generators of $J(G)$ correspond to the
minimal vertex covers of $G$.  A subset $W \subseteq V$ is a {\it vertex
cover} of $G$ if $e \cap W \neq \emptyset$ for all $e \in E$.
Edge and cover ideals give us an algebraic way to study graphs;
for more on these ideals and their properties,
see \cite{HVT2007,MV2012,SF2020,VT2013}.
By
restricting to cover and edge ideals, one can exploit the extra
combinatorial information.

\begin{example}\label{runningexamplepic}
  Consider the graph $G$ on five vertices as given in Figure \ref{graphpicture}.
  For this graph, the edge ideal is
  $$I(G) = \langle x_1x_2,x_1x_3,x_2x_3,x_3x_4,x_3x_5 \rangle,$$
  and the cover ideal is
  \begin{eqnarray*}
    J(G) &=&\langle x_1,x_2 \rangle \cap \langle x_1,x_3 \rangle
    \cap \langle x_2,x_3 \rangle \cap \langle x_3,x_4 \rangle \cap \langle
    x_3,x_5 \rangle =  \langle x_1x_3,x_2x_3,x_1x_2x_4x_5 \rangle.
  \end{eqnarray*}
  Both ideals belong to the ring $R = k[x_1,\ldots,x_5]$.
  Observe that $J(G)$ is our running example.  Also, each generator
  of $J(G)$ corresponds to a vertex cover of $G$.  For example,
  if we look at the generator $x_1x_3$, then every edge of $G$ has either
  $x_1$ or $x_3$ as an endpoint.
\begin{figure}[!ht]
    \centering
    \begin{tikzpicture}[scale=0.45]
      \draw (0,0) -- (5,0);
      \draw (0,0) -- (2.5,4);
      \draw (2.5,4) -- (5,0);
      \draw (5,0) -- (7.5,4);
      \draw (5,0) -- (10,0);

      \fill[fill=white,draw=black] (0,0) circle (.1) node[below]{$x_1$};
      \fill[fill=white,draw=black] (5,0) circle (.1) node[below]{$x_3$};
      \fill[fill=white,draw=black] (10,0) circle (.1) node[below]{$x_4$};
      \fill[fill=white,draw=black] (7.5,4) circle (.1) node[above]{$x_5$};
      \fill[fill=white,draw=black] (2.5,4) circle (.1) node[above]{$x_2$};
     \end{tikzpicture}
    \caption{The graph $G$}
    \label{graphpicture}
 \end{figure}
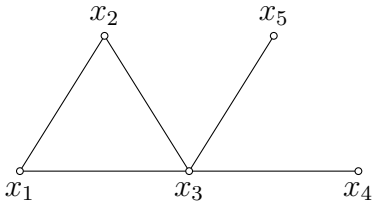
\end{example}

It turns out that the cover ideals of chordal graphs provide
a large family of componentwise ideals. Given a graph
$G = (V,E)$, the {\it induced subgraph} on the set $W \subseteq V$
is the graph $G_W = (W,E_W)$ where $E_W = \{e \in E ~|~ e \subseteq W\}$,
that is, an edge $e$ of $G$ also belongs to $G_W$ if and only if
both endpoints of $e$ belong to $W$.  A {\it cycle of length $n$}
is the graph
$$C_n = (\{x_1,\ldots,x_n\},\{\{x_1,x_2\},\{x_2,x_3\},\ldots,\{x_{n-1},x_n\},
\{x_n,x_1\}\}).$$
A graph $G$ is a {\it chordal graph} if $G$ has no induced
subgraphs isomorphic to a $C_n$ with $n \geq 4$.  The graph
in Example \ref{runningexamplepic} is an example of chordal graph.

The following result is due to Francisco and Van Tuyl:

\begin{theorem}[{\cite[Theorem 3.2]{FVT2007}}]\label{chordaliscwl}
  If $G$ is a chordal graph, then the cover ideal $J(G)$
  is componentwise linear.
\end{theorem}

Note that the statement in \cite{FVT2007} uses the equivalence of
Theorem \ref{scmstatement}, that is, the conclusion is that
$\Delta^\vee_{J(G)}$ is a sequentially Cohen-Macaulay simplicial complex.
To prove Theorem \ref{chordaliscwl},
Francisco and Van Tuyl use Theorem \ref{sqfreecwl} to
show that $J(G)$ is square-free componentwise linear.
Other proofs of Theorem \ref{chordaliscwl} exist making
use of Theorem \ref{vd=>shellable=>scm}: (1)
Van Tuyl and Villarreal \cite{VTV2008}
showed that $\Delta_{J(G)}^\vee$ is shellable, from which we can deduce that
$J(G)$ is componentwise linear;
(2) Erey \cite{E2019a} gave an alternative ordering of the generators
of $J(G)$ to prove that $J(G)$ has
the linear quotients; and (3) Dochtermann and Engstr\"om \cite{DE2009},
and independently
Woodroofe \cite{W2009}, proved that
$\Delta_{J(G)}^\vee$ was vertex decomposable when
$G$ is chordal, which implies $\Delta_{J(G)}^\vee$ is shellable.

Herzog, Hibi, and Ohsugi \cite{HHO2011} were the first to
consider the problem of when powers of componentwise linear ideals are
also componentwise linear.   Using Theorem \ref{chordaliscwl} as their
starting point, they proposed the following conjecture
\cite[Conjecture 2.5]{HHO2011}.

\begin{conjecture}[The Herzog-Hibi-Ohsugi Conjecture]\label{hhoconjecture}
  If $G$ is a chordal graph with cover ideal $J(G)$,
  then $J(G)^s$ is componentwise linear for all $s \geq 1$.
\end{conjecture}

Part of the difficulty of this conjecture lies in the fact that
$J(G)^s$ is no longer a square-free monomial ideal if $s \geq 2$.  The
majority of the proofs for Theorem \ref{chordaliscwl} as described above
rely heavily on the fact that $J(G)$ is a square-free monomial ideal.  Answering
the Herzog-Hibi-Ohsugi Conjecture will rely on new techniques for
proving a monomial ideal is componentwise linear.

There is a growing body of families of chordal graphs that
satisfy Conjecture \ref{hhoconjecture}, thus pointing towards the validity
of the conjecture.   The proof strategies broadly
fall into two categories.  One strategy is to find
a description or ordering of the generators of $J(G)^s$,
and then show that the ideal has linear quotients
to apply Theorem \ref{thm.LQimpliesCWL}, or $J(G)^s$
satisfies a property like being weakly polymatroidal and apply a result like
Theorem \ref{thm.weaklypoly}.  The second strategy is to use
properties of the Rees algebra,
as in Theorem \ref{linearpowers}.  Current attacks on the conjecture
have also exploited the structure of chordal graphs.  For example,
some approaches exploit chordal graphs with ``lots'' of edges, e.g.,
a graph with a large complete graph (defined below) as a subgraph.  At
the other extreme, the conjecture has been investigated for
chordal graphs with ``few'' edges,
e.g., trees or graphs with very rigid structure.

Herzog, Hibi, and Ohsugi provided the initial
evidence for the validity of Conjecture \ref{hhoconjecture}.
Under the extra assumption that $J(G)$ has a linear resolution
(which is equivalent to the fact that $\Delta_{J(G)}^\vee$ is a Cohen-Macaulay
simplicial complex by Theorem \ref{eagonreiner}), they show
that ${\rm reg}_{(1,0)}R(J(G)) = 0$, and use the approach
of Theorem \ref{linearpowers} to verify Conjecture \ref{hhoconjecture}
for all cover ideals of chordal graphs with a linear resolution.

\begin{theorem}[{\cite[Theorem 2.7]{HHO2011}}]\label{chordallinear}
  If $G$ is
  a chordal graph such that cover ideal $J(G)$ has a linear resolution,
  then $J(G)^s$ is componentwise linear for all $s \geq 1$.
\end{theorem}

We need to introduce some
special classes of chordal graphs.
The {\it complete graph} $K_n$ on $n$ vertices is the graph with
vertex set $\{x_1,\ldots,x_n\}$ and edge set $\{\{x_i,x_j\} ~|~ 1 \leq
i < j \leq n\}$.  We call a graph $G$ on $n+m$ vertices
a {\it star graph} based
on $K_n$ if the vertices of $G$ can be relabeled so that
the induced graph on $\{x_1,\ldots,x_n\}$ is the complete
graph $K_n$, and for any $n \leq i < j \leq n+m$, the
edge $\{x_i,x_j\} \not\in E$.  The graph in Figure \ref{runningexamplepic}
  is an example of a star graph based on $K_3$ since the induced
  graph on $\{x_1,x_2,x_3\}$ is a $K_3$.

  Mohammadi \cite{M2011} introduced a wider class of ideals that
  generalized this construction that
  were called {\it generalized star graphs}.  While we will not
  recall this construction, the idea is similar in that one
  glues together a collection of complete graphs in a prescribed way
  to form a ``core'', and then one is allowed to attach some extra
  edges.  We then have the following result.

  \begin{theorem}\label{stargraphs}
    Conjecture \ref{hhoconjecture} is true
    for the cover ideals of the following chordal graphs:
  \begin{enumerate}
  \item \cite[Theorem 2.3]{HHO2011} Star graphs based on $K_n$.
  \item \cite[Theorem 1.5]{M2011} Generalized star graphs.
  \end{enumerate}
  \end{theorem}

  \begin{example}
    For any complete graph $K_n$, the cover ideal $J(K_n)$ has a
    linear resolution.  This can be checked directly from the
    fact that $J(K_n) = \langle x_1\cdots \widehat{x_i} \cdots x_n ~|~
    1 \leq i \leq n \rangle$, and that this ideal has linear quotients.
    So all powers of $J(K_n)$ are componentwise linear.

    As a second example, the graph in Figure \ref{runningexamplepic} is a
    star graph based on $K_3$.  Consequently, the ideal of
    our running example, that is,
    $J(G)  = \langle x_1x_3,x_2x_3,x_1x_2x_4x_5 \rangle$
    has the property that all of its powers has the componentwise linear
    property.
  \end{example}

  Herzog, Hibi, Ohsugi's proof of Theorem \ref{stargraphs}
  uses properties of the Rees algebra $R(J(G))$.  Mohammadi
  gives a different proof for star graphs based on $K_n$ that shows
  for each graph $G$ in this family, the ideal
  $J(G)^s$ is weakly polymatroidal.    In fact, the following
  theorem gives a combinatorial way to check if all powers
  of $J(G)$ are componentwise linear.   Given a graph $G = (V,E)$,
  the {\it clique complex} of $G$ is the simplicial complex
  $${\rm Cliq}(G) = \langle W \subseteq V ~|~ \mbox{$G_W = K_{|W|}$} \rangle,$$
  in other words, the clique complex consists of all the subsets
  of $V$ such that induced graph on that subset is a complete graph.
  For any simplicial complex $\Delta$, a facet $F \in \Delta$ has
  a {\it free vertex} if there is some vertex $x_i \in F$ that only appears
  in $F$, but no other facet.  We denote the set of all facets of $\Delta$ with
  a free vertex by $\mathcal{FV}(\Delta)$.  We then have the
  following tool:

  \begin{theorem}[{\cite[Corollary 2.4]{M2014}}]\label{combinatorialresult}
    Let $G = (V,E)$ be a chordal graph
    with clique complex ${\rm Cliq}(G)$.
    If $|V|-1 \leq | \bigcup_{F\in \mathcal{FV}({\rm Cliq}(G))} F |$,
    then the ideal $J(G)^s$ is componentwise linear for all $s \geq 1$.
    \end{theorem}

  \begin{example}
    The graph $G$ in Figure \ref{graphpicture} has clique complex
    $${\rm Cliq}(G) = \langle \{x_1,x_2,x_3\},\{x_3,x_4\},\{x_3,x_5\} \rangle.$$
    Note that each facet has a free vertex:  $x_1,x_2$ only appear in
    the first facet, while $x_4$ only appears in the second, and $x_5$ only
    in the third.  Since
    $$|V|-1 = 4 \leq |\{x_1,x_2,x_3\}\cup \{x_3,x_4\} \cup \{x_3,x_5\}| = 5.$$
    all powers of the cover ideal $J(G)$ are componentwise linear
    by Theorem \ref{combinatorialresult}.
    \end{example}

  Erey (see \cite{E2019,E2021}) approached Conjecture
  \ref{hhoconjecture} by finding an order of the minimal generators
  of $J(G)^s$ that gives linear quotients.  In the statement
  below, the {\it path graph} $P_n$ is
  the graph with vertex set $V = \{x_1,\ldots,x_n\}$
  and edge set $E = \{\{x_1,x_2\},\{x_2,x_3\},\ldots,\{x_{n-1},x_n\}\}$.
  A graph is {\it $(C_4,2K_2)$-free} if it has no induced graph
  isomorphic to $C_4$ or two copies of $K_2$.

\begin{theorem} Conjecture \ref{hhoconjecture} is true
  for the cover ideals of the following chordal graphs:
  \begin{enumerate}
  \item \cite[Theorem 3.7]{E2019} Chordal graphs
    that are also $(C_4,2K_2)$-free graphs.
  \item \cite[Theorem 4.3]{E2021} The path graphs $P_n$.
  \end{enumerate}
  \end{theorem}

\begin{remark}
Erey showed a stronger result in
    \cite[Theorem 3.7]{E2019}, namely, the
    cover ideal $J(G)$ of {\it any} $(C_4,2K_2)$-free graph $G$ satisfies
    the property that $J(G)^s$ is componentwise linear for all $s \geq 1$.
    Note that the cycle $C_5$ is a $(C_4,2K_2)$-free graph that is not
    chordal.
  \end{remark}

\begin{remark}
  As an intermediate step, Erey and Qureshi first proved
  that $J(P_n)^2$ was componentwise linear in
  \cite[Theorem 5.1]{EQ2019}.  Erey was later able to
  extend this result to all powers, as noted above.
\end{remark}

Herzog, Hibi, and Moradi were also able to prove
that the same result for $P_n$ as a consequence
of a more general result, which again uses the
Rees algebra.  We recall
how one can construct the Rees algebra of an ideal $I = \langle f_1,\ldots,
f_s \rangle \subseteq R = k[x_1,\ldots,x_n]$ when $I$ is not necessarily
generated by terms of the same degree.  In the ring $R[t]$, consider
the subring
$$R[It] = k[f_1t,f_2t,\ldots,f_st] \subseteq R[t].$$
Let $S = k[y_1,\ldots,y_s,x_1,\ldots,x_n]$.  We define a
$k$-algebra homomorphism $\varphi:S \rightarrow R$ by
$$y_i \mapsto f_it ~~\mbox{and}~~ x_j \mapsto x_j $$
for $i=1,\ldots,s$ and $j=1,\ldots,n$.  Let $J = {\rm ker}{\varphi}$.
The ideal $J$ is  called the {\it defining ideal of the Rees ring}
${\rm Rees}(I)$.  A new criteria for when an ideal is componentwise linear
is then given in terms of $J$.

\begin{theorem}[{\cite[Theorem 2.6]{HHM2020}}]\label{quadraticrees}
    Let $I$ be a monomial ideal and let $J$ be the defining ideal of the Rees
    ring ${\rm Rees}(I)$.  There exists a monomial order
    $>$ such that if the initial ideal ${\rm in}_>(J)$
    is generated by quadratic monomials with respect to this order,
    then $I^s$ is componentwise
    linear for all $s \geq 1$.
\end{theorem}

The definition of the required monomial order can be found in \cite{HHM2020}.
Using Theorem \ref{quadraticrees}, Herzog, Hibi, and Moradi
were able to verify Conjecture \ref{hhoconjecture} for more families
of chordal families.  A
{\it biclique graphs} is a graph on the vertex set
$\{x_1,\ldots,x_p,y_1,\ldots,y_q,z_1,\ldots,z_r\}$ such that
the induced graphs on $\{x_1,\ldots,x_p,y_1,\ldots,y_q\}$ and
on $\{y_1,\ldots,y_q,z_1,\ldots,z_r\}$ are complete graphs.
{\it Cameron-Walker graphs} are graphs whose induced matching
number equals its matching number; we do not
formally define this family here, but point the reader to
\cite{HHKO2015}.

\begin{theorem}[{\cite[Corollary 4.7]{HHM2020}}]
  Conjecture \ref{hhoconjecture} is true
  for the cover ideals of the following chordal graphs:
  \begin{enumerate}
  \item Biclique graphs.
  \item The path graphs $P_n$.
  \item Cameron-Walker graphs whose bipartite graph
    is a complete bipartite graph.
  \end{enumerate}
\end{theorem}

Kumar and Kumar have recently shown that Conjecture \ref{hhoconjecture}
holds for all trees.  {\it Trees} are graphs which have no induced cycles, and
thus, they are examples of chordal graphs.  Kumar and Kumar's proof
uses a different strategy then the above results.  In the case
that $G$ is a tree, $J(G)^s$ equals its $s$-symbolic power (to be
defined in the next section).  It can the be shown that
$(J(G)^s)^{\rm pol}$,
the polarization of $J(G)^s$, is the cover ideal of another graph.
The proof for trees then shows that the cover ideal of this new graph
is also componentwise linear.  This strategy will be expanded upon
in more detail when we look at symbolic powers in the next section.

\begin{theorem}[{\cite[Corollary 3.5]{KK2020}}]
     Conjecture \ref{hhoconjecture} is true
     for the cover ideals of all trees.
\end{theorem}

\begin{remark}
  The above result is slightly stronger since it is shown
  that $J(G)^s$ has linear quotients for all $s \geq 0$ when $G$ is a
  tree.
  \end{remark}

We round out this section by describing three
results not directly related to Conjecture \ref{hhoconjecture},
but related to the more general theme of Question \ref{maintheme}.
The first result concerns the cover ideals of graphs that may not
be chordal.  The second result concerns edge ideals, not cover ideals,
and the third is for quadratic monomial ideals (which includes all edge ideals).

A  graph $G$ is {\it bipartite} if the vertex
set of $G$ can partitioned as $V = V_1 \cup V_2$ so that for
every edge $e \in E$, one has $e \cap V_1 \neq \emptyset$
and $e \cap V_2 \neq \emptyset$.  That is, every edge has one
endpoint in $V_1$ and the other in $V_2$.  We then
have the following result for the cover ideals of bipartite
graphs; an earlier version of this result appeared in \cite[Theorem 2.2]{MM2010}.

\begin{theorem}[{\cite[Corollary 3.7]{SF2018}}] \label{bipartitelinear}
  Let $G$ be a bipartite graph.  Then $J(G)$ has linear
  powers if and only if $J(G)$ has a linear resolution.
\end{theorem}

\noindent

As is evident from this section, the majority of work
on powers of componentwise linear ideals has focused on cover ideals.
Of course, similar questions could be asked about edge ideals.
To-date, the primary focus has been on the linear property, as demonstrated
in the next two results.

\begin{theorem}[{\cite[Theorem 2.12]{E2019}}] Let $G$ be a
  $(C_4,2K_2)$-free graph with edge ideal $I(G)$.  Then
  $I(G)^s$ is componentwise linear for all $s \geq 2$;
  in fact, $I(G)^s$ has a linear resolution for all $s \geq 2$.
\end{theorem}

Note that in the above result the ideal $I(G)$ may or may not be componentwise
linear, but its powers are.
Our final result looks at quadratic monomial ideals that
need not be square-free.

\begin{theorem}[{\cite[Theorem 3.2]{HHZ2004}}]
  Let $I$ be a quadratic monomial ideal.  Then $I$ has linear
  powers if and only if $I$ has a linear resolution.
\end{theorem}


\section{Componentwise linearity of symbolic powers}

In this section, we move beyond the Herzog-Hibi-Ohsugi conjecture to address
the question of when \emph{symbolic} powers of an ideal are componentwise
linear. For any arbitrary ideal $I \subseteq R$, the {\it $s$-th symbolic
  power} of $I$ is the ideal
$$I^{(s)} = \bigcap_{P \in {\rm Ass}(I)} (I^sR_P \cap R)$$
where ${\rm Ass}(I)$ is the set of associated primes of $I$ and $R_P$
is the ring $R$ localized at the prime ideal $P$.  In the case that
$I$ is a square-free monomial ideal with primary decomposition
$I = P_1 \cap \cdots \cap P_r$, its $s$-th symbolic power is
given by
$$I^{(s)} = P_1^s \cap \cdots \cap P_r^s.$$
In particular, the $s$-th symbolic power of the cover ideal of a
graph $G$ satisfies
$$J(G)^{(s)} = \bigcap_{\{x_i,x_j\} \in E} \langle x_i,x_j \rangle^s.$$
We survey a number of recent articles focusing on the class of
cover ideals of graphs that have addressed the following
umbrella question:

\begin{question} \label{quest.SP}
For which graphs $G$ is $J(G)^{(s)}$ componentwise linear for all $s \ge 1$?
\end{question}

Going forward, we will employ the following terminology. For a graph $G$,
its \emph{independent complex}, denoted by $\text{Ind}(G)$, is the
simplicial complex whose faces are independent sets in $G$.
A set $A \subseteq V$ is an {\it independent set} if for all
$e \in E$, $e \not\subseteq A$.  Equivalently, $V \setminus A$ is
a vertex cover. It is not
hard to show that $\text{Ind}(G) = \Delta_{J(G)}^\vee$.

\begin{definition}
  A graph $G$ is {\it vertex decomposable}, respectively {\it shellable},
  if its independent complex ${\rm Ind}(G)$ is vertex decomposable,
  respectively shellable.
  \end{definition}

\noindent
Note that by Theorem \ref{vd=>shellable=>scm}, if a graph
$G$ is vertex decomposable or shellable, then $J(G)$ is componentwise
linear.

When $G$ is a bipartite graph, it is known (cf. \cite{GRV2005}) that
$J(G)^s = J(G)^{(s)}$ for all $s \ge 1$, and so Question \ref{quest.SP}
reduces to the question of when regular powers of the cover ideal of a graph
are componentwise linear --- this question has been discussed in the previous
section and is closely related to the Herzog-Hibi-Ohsugi conjecture. In this
case, by combining previous work of Seyed Fakhari \cite{SF2018} and of
Selvaraja and Skelton \cite{SS2021}, one obtains the following result.

\begin{theorem}[{\cite[Theorem 3.6 and Corollary 3.7]{SF2018} and \cite[Theorem 5.3]{SS2021}}] Let $G$ be a bipartite graph, and thus $J(G)^s = J(G)^{(s)}$.
  \begin{enumerate}
  \item The following are equivalent:
	\begin{enumerate}
		\item $J(G)^s$ is componentwise linear for all $s \geq 1$,
		\item $J(G)^s$ is componentwise linear for some $s > 1$,
		\item $J(G)^s$ has linear quotients for all $s \geq 1$,
		\item $G$ is a vertex decomposable graph.
	\end{enumerate}
      \item
        The following are equivalent:
	\begin{enumerate}
		\item $J(G)^s$ has a linear resolution for all $s \geq 1$,
		\item $J(G)$ has a linear resolution,
		\item $G$ is a pure vertex decomposable graph (i.e.,
                  ${\rm Ind}(\Delta) =\Delta^\vee_{J(G)}$ is also a pure simplicial complex)
	\end{enumerate}
        \end{enumerate}
\end{theorem}

For an arbitrary graph $G$, the general approach to investigate symbolic powers of the cover ideal $J(G)$ is to view the polarization of these symbolic powers as the cover ideals of other graphs constructed from $G$.
Particularly, the following constructions, due to Seyed Fakhari \cite{SF2018} and Kumar and Kumar \cite{KK2020}, have proved to be essential in this line of work.

\begin{construction}[Duplicating vertices] \label{const.vertices}
  Let $G$ be a graph over the vertex set $V_G = \{x_1, \dots, x_n\}$
  and let $s \ge 1$ be an integer. We construct a new graph,
  denoted by $G_s$, as follows:
	\begin{align*}
	  V_{G_s} & = \{x_{i,p} ~\big|~ 1 \le i \le n, \ 1 \le p \le s\}, \text{ and } \\
	  E_{G_s} & = \left\{ \{x_{i,p}, x_{j,q}\} ~\big|~ \{x_i, x_j\} \in E_G \text{ and } p+q \le s+1\right\}.
	\end{align*}
\end{construction}

\begin{construction}[Duplicating edges]
	\label{const.edges}
	Let $G$ be a graph with vertex set $V_G = \{x_1, \dots, x_n\}$ and edge set $E_G = \{e_1, \dots, e_m\}$.
	\begin{enumerate}
		\item Let $r \in \ZZ_{\ge 0}$ and $e = \{x_i,x_j\} \in E_G$. Set
		\begin{align*}
			V(e(r)) & = \{x_{l,p} ~\big|~ l \in \{i,j\} \text{ and } 1 \le p \le r\}, \text{ and } \\
			E(e(r)) & = \left\{ \{x_{i,p}, x_{j,q}\} ~\big|~ p+q \le r+1\right\}.
		\end{align*}
	\item For an ordered tuple $(s_1, \dots, s_m) \in \ZZ^m_{\ge 0}$, we construct a new graph, denoted by $G(s_1, \dots, s_m)$, as follows:
	\begin{align*}
		V_{G(s_1, \dots, s_m)} & = \bigcup_{i=1}^m V(e_i(s_i)), \text{ and } \\
		E_{G(s_1, \dots, s_m)} & = \bigcup_{i=1}^m E(e_i(s_i)).
	\end{align*}
	\end{enumerate}
\end{construction}
Obviously, for $s_1 = \dots = s_m = s$, we have $G(s_1, \dots, s_m) = G_s$.
The use of Constructions \ref{const.vertices} and \ref{const.edges} is
reflected in the following lemma.

\begin{lemma}[{\cite[Lemma 3.4]{SF2018}}]
	\label{lem.SF2018}
	Let $G$ be a graph and let $J(G)$ be its cover ideal. For any integer $s \ge 1$, the polarization $\left(J(G)^{(s)}\right)^{\pol}$ coincides with the cover ideal of $G_s$.
\end{lemma}

\begin{example}
  We illustrate the above ideas by using the graph of Example \ref{runningexamplepic} for $s = 2$.  The graph $G_2$ is then given in Figure
  \ref{graphpicture2}.
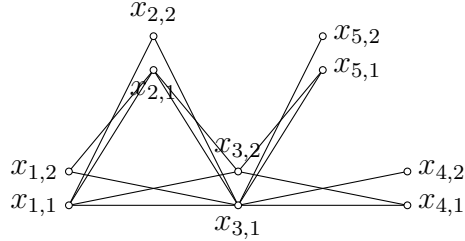
\begin{figure}[!ht]
    \centering
    \begin{tikzpicture}[scale=0.45]
       \draw (0,0) -- (5,1);
       \draw (0,0) -- (5,0);
       \draw (0,0) -- (2.5,5);
       \draw (0,0) -- (2.5,4);
       \draw (2.5,4) -- (5,0);
       \draw (2.5,4) -- (5,1);
       \draw (5,0) -- (7.5,4);
       \draw (5,0) -- (7.5,5);
       \draw (5,0) -- (10,0);
       \draw (5,0) -- (10,1);
       \draw (0,1) -- (5,0);
       \draw (0,1) -- (2.5,4);
       \draw (2.5,5) -- (5,0);
       \draw (5,1) -- (7.5,4);
       \draw (5,1) -- (10,0);

      \fill[fill=white,draw=black] (0,0) circle (.1) node[left]{$x_{1,1}$};
      \fill[fill=white,draw=black] (5,0) circle (.1) node[below]{$x_{3,1}$};
      \fill[fill=white,draw=black] (10,0) circle (.1) node[right]{$x_{4,1}$};
      \fill[fill=white,draw=black] (7.5,4) circle (.1) node[right]{$x_{5,1}$};
      \fill[fill=white,draw=black] (2.5,4) circle (.1) node[below]{$x_{2,1}$};
      \fill[fill=white,draw=black] (0,1) circle (.1) node[left]{$x_{1,2}$};
      \fill[fill=white,draw=black] (5,1) circle (.1) node[above]{$x_{3,2}$};
      \fill[fill=white,draw=black] (10,1) circle (.1) node[right]{$x_{4,2}$};
      \fill[fill=white,draw=black] (7.5,5) circle (.1) node[right]{$x_{5,2}$};
      \fill[fill=white,draw=black] (2.5,5) circle (.1) node[above]{$x_{2,2}$};
     \end{tikzpicture}
    \caption{The graph $G_2$ constructed from the graph $G$ in Figure \ref{graphpicture}}
    \label{graphpicture2}
\end{figure}
Note that $J(G) = \langle x_1x_3,x_2x_3,x_1x_2x_4x_5 \rangle$, so
the ideal $J(G)^{(2)}$ is given by
$$J(G)^{(2)} = \langle x_1^2x_3^2, x_1x_2x_3^2,
x_2^2x_3^2,x_1x_2x_3x_4x_5,x_1^2x_2^2x_4^2x_5^2 \rangle.$$
The polarization of $J(G)^{(2)}$ is then the ideal
\begin{eqnarray*}
(J(G)^{(2)})^{\rm pol} &=& \langle
x_{1,1}x_{1,2}x_{3,1}x_{3,2},x_{1,1}x_{2,1}x_{3,1}x_{3,2},x_{2,1}x_{2,2}x_{3,1}x_{3,2},\\
&& x_{1,1}x_{2,1}x_{3,1}x_{4,1}x_{5,1},
x_{1,1}x_{1,2}x_{2,1}x_{2,2}x_{4,1}x_{4,2}x_{5,1}x_{5,2} \rangle.
\end{eqnarray*}
This ideal then satisfies $(J(G)^{(2)})^{\rm pol} = J(G_2)$.
\end{example}

Lemma \ref{lem.SF2018} fits into the context of studying the
componentwise linearity of symbolic powers of the cover ideal $J(G)$
via the following result. It allows us to, instead of looking at the
componentwise linearity of $J(G)^{(s)}$, consider when $G_s$ is vertex
decomposable, which is a combinatorial property and could be more natural
to examine.

\begin{lemma}
	\label{lem.VDimpliesCWL}
	Let $G$ be a graph and let $s \in \NN$. If $G_s$ is vertex
        decomposable, then $J(G_s)$ has linear quotients.
        Particularly, if $G_s$ is vertex decomposable,
        then $J(G)^{(s)}$ has linear quotients and is componentwise linear.
\end{lemma}

\begin{proof}
  By Theorem \ref{vd=>shellable=>scm}, we know that if $G_s$ is
  vertex decomposable, then $J(G_s)$ has linear quotients.
  Thus, Lemma \ref{lem.SF2018}
  implies $(J(G^{(s)})^{\rm pol}$ has linear quotients.
  This, together with Theorem \ref{linquot-polar}
  implies that $J(G)^{(s)}$ has linear quotients. The last statement follows from Theorem \ref{thm.LQimpliesCWL}.
\end{proof}

By applying Lemma \ref{lem.VDimpliesCWL}, Seyed Fakhari \cite{SF2018}, Selvaraja \cite{S2000}, and Kumar and Kumar \cite{KK2020} showed that
the following special classes of graphs $G$ have the property that
$J(G)^{(s)}$ is componentwise linear for any $s \ge 1$.

\begin{theorem}
	\label{thm.CWLclasses}
	Let $G$ be a graph.
	\begin{enumerate}
	\item {\cite[Theorem 3.6]{SF2018}} If $G$ is very
          well-covered and $J(G)$ has a linear resolution,
          then $J(G)^{(s)}$ has linear quotients for all $s \ge 1$.
	\item {\cite[Corollary 4.7]{S2020}} If $G$ is a
          Cameron-Walker graph,
          then $J(G)^{(s)}$ has linear quotients for all $s \ge 1$.
	\item {\cite[Theorem 3.4 and Corollary 3.5]{KK2020}} If $G$ is a
          tree on $n$ vertices, then for any tuple
          $(s_1, \dots, s_{n-1}) \in \ZZ^{n-1}_{\ge 0}$, $G(s_1, \dots, s_{n-1})$
          is a vertex decomposable graph.
          Particularly, $J(G)^s = J(G)^{(s)}$ has linear quotients for all
          $s \ge 1$.
	\item {\cite[Corollary 4.5]{KK2020}} If $G$ is a uni-cyclic
          vertex decomposable graph, then
          $J(G)^{(s)}$ is componentwise linear for all $s \ge 1$.
	\end{enumerate}
\end{theorem}
\noindent
In the above statement, a graph is {\it uni-cyclic} if the graph
has only one induced cycle, and  a graph $G$
is {\it very well-covered} if all of its maximal vertex covers
have cardinality $\frac{1}{2}|V|$.

Seyed Fakhari \cite{SF2021} improved his previous result \cite[Theorem 3.6]{SF2018} (see Theorem \ref{thm.CWLclasses} (1)). In particular, he classified all the graphs whose symbolic powers have a linear resolution.

  \begin{theorem}[{\cite[Theorem 3.4]{SF2021}}] Let
    $G$ be a graph with no isolated vertices.  Then the following
    are equivalent:
    \begin{enumerate}
    \item $J(G)^{(s)}$ has a linear resolution for all $s \geq 1$,
    \item $J(G)^{(s)}$ has a linear resolution for some $s \geq 1$, and
    \item $G$ is very well-covered and  ${\rm Ind}(G)$ is Cohen-Macaulay.
    \end{enumerate}
    \end{theorem}

In a different approach, Selvaraja and Skelton \cite{SS2021} gave the following sufficient condition for $J(G)^{(s)}$ to fail componentwise linearity for all $s \ge 1$.

\begin{theorem}[{\cite[Theorem 3.1]{SS2021}}]
	\label{thm.nonCWLSS}
	Let $G$ be graph, and suppose that $J(G)^{(s)}$ is not componentwise
	linear for either $s=1$ and $2$, or $s=2$ and $3$.  Then
	$J(G)^{(s)}$ is not componentwise linear for all $s \geq 2$.
\end{theorem}

The strategy to prove Theorem \ref{thm.nonCWLSS} is to use Seyed Fakhari's construction of $G_s$, and then show that in the inductive hypothesis of $G_s$ being vertex decomposable, a subgraph obtained from $G_s$ by removing the neighbours of a shedding vertex is isomorphic to $G_{s-2}$.

In the same spirit, finding conditions so that $J(G)^{(s)}$ fails to be componentwise linear for all $s \ge 1$, Selvaraja and Skelton \cite{SS2021} gave the following result. Note that for a vertex decomposable graph $G$ with a shedding sequence $x_{\alpha(1)}, \dots, x_{\alpha(l)}$, let $\{x_{\gamma(1)}, \dots, x_{\gamma(r)}\}$ be the collection of isolated vertices remaining in $G \setminus \{x_{\alpha(1)}, \dots, x_{\alpha(l)}\}$. The \emph{spanning bipartite graph} $\B_G$ is defined to be the bipartite graph with the bipartition of the vertices $\{x_{\alpha(1)}, \dots, x_{\alpha(l)}\} \cup \{x_{\gamma(1)}, \dots, x_{\gamma(r)}\}$ and edges
$$\{\{x_{\alpha(i)},x_{\gamma(j)}\} ~\big|~ \{x_{\alpha(i)},x_{\gamma(j)}\} \in E_G, 1 \le i \le l \text{ and } 1 \le j \le r\}.$$

\begin{theorem}[{\cite[Theorem 3.6]{SS2021}}]
	\label{thm.nonCWL_B}
	Let $G$ be a vertex decomposable graph.  If there
	exists an independent set $A$ such that
	$\mathcal{B}_{G\setminus N[A]}$ is not vertex decomposable,
	then $J(G)^{(s)}$ is not componentwise linear for all $s \geq 2$.
\end{theorem}

The necessary condition in Theorem \ref{thm.nonCWL_B} is also sufficient to achieve the componentwise linearity of $J(G)^{(s)}$, for all $s \ge 1$, for a special class of vertex decomposable graphs, namely, the class of $W$-graphs. Selvaraja and Skelton \cite{SS2021} defined a \emph{$W$-graph} $G$ to be graph such that $G \setminus N[A]$ has a simplicial vertex for any independent set $A$.
A vertex $x$ is a {\it simplicial vertex} if the induced
graph on $x$ and all of its neighbors is a complete graph.

\begin{theorem}[{\cite[Theorem 4.2]{SS2021}}] Let $G$ be a $W$-graph.  Then
	the following are equivalent:
	\begin{enumerate}
		\item $\mathcal{B}_{G \setminus N[A]}$ is vertex decomposable
		for any independent set $A$ in $G$.
		\item $J(G)^{(s)}$ is componentwise linear for all $s \geq 1$.
		\item $J(G)^{(s)}$ is componentwise linear for some $s \geq 2$.
	\end{enumerate}
\end{theorem}

In addressing Question \ref{quest.SP} and identifying new classes of
graphs for which all symbolic powers of the cover ideal are
componentwise linear, the following approach has been
investigated: combinatorially modify a given graph $G$ to obtain a
new graph $G'$ with the required property that $J(G')^{(s)}$ is
componentwise linear for any $s \ge 1$. Specifically, originating
from Villarreal's work \cite{V1990}, the process of adding
\emph{whiskers} (or \emph{whiskering}) to the vertices of a graph
has been studied and developed by many authors and from various
directions
(cf.  \cite{BVT2013, BFHVT2015, CN2012, DHNT2021, FH2008, GHS2021, KK2020, S2020, SS2021}).

\begin{definition}
	\label{def.whiskering}
	By adding a \emph{whisker} to a vertex $x$ of a graph $G$, one adds a new vertex $y$ and the edge $\{x,y\}$ to $G$. Let $S \subseteq V_G$ be a subset of the vertices in $G$. Then we denote by $G \cup W(S)$ the graph obtained by adding a whisker to $G$ at each vertex in $S$.
\end{definition}

The first result in this approach to Question \ref{quest.SP} is due to Dung, Hien, Nguyen and Trung \cite{DHNT2021}, which shows that by adding a whisker to every vertex of any given graph one obtains a new graph with the desired property. The case $s=1$ of the following theorem in fact implies Villarreal's result in \cite{V1990}.

\begin{theorem}[{\cite[Corollary 5.9]{DHNT2021}}] \label{thm.DHNT}
  Let $G$ be a graph and let $H = G \cup W(V_G)$ be the graph obtained by
  adding a whisker at every vertex in $G$. Then $J(H)^{(s)}$ is
  componentwise linear for all $s \ge 1$.
\end{theorem}
\noindent
Dung, Hien, Nguyen, and Trung, in fact, proved a stronger statement than Theorem \ref{thm.DHNT} in \cite[Theorem 5.7]{DHNT2021}, where they
  showed that the same conclusion holds if at least one whisker is added to every vertex of $G$.

Cook and Nagel \cite{CN2012} generalized the process of whiskering to that of \emph{clique-whiskering} to extend Villarreal's previous work \cite{V1990}.

\begin{definition}
  Let $G$ be a graph. A \emph{clique partition} $\pi$ of $G$ is a partition of the vertices in $G$ into disjoint (possibly empty) subsets $W_1, \dots, W_t$ such that the induced graphs $G_{W_i}$  is a complete (or empty) graph in $G$ for all $i = 1, \dots, t$. A \emph{clique-whiskering} of $G$ associated to a clique-partition $\pi$, denoted by $G^\pi$, is a graph over the
  vertices $V_{G^\pi} = V_G \cup \{y_1, \dots, y_t\}$ and has edges
	$$E_{G^\pi} = E_G \cup \left(\bigcup_{i=1}^t \left\{\{v,y_i\} ~\big|~ v \in W_i\right\}\right).$$
\end{definition}

Selvaraja \cite{S2020} proved the following theorem, of which the case where $s=1$ was known in the previous work of Cook and Nagel \cite{CN2012}.

\begin{theorem}[{\cite[Theorem 4.9]{S2020}}] \label{thm.Selva}
	Let $G$ be a graph and let $\pi$ be a clique vertex partition of $G$. Then $J(G^\pi)^{(s)}$ has linear quotients for any $s \ge 1$.
\end{theorem}

Inspired by Theorem \ref{thm.DHNT}, the following question arises naturally: for which subset $S \subseteq V_G$ of the vertices in a graph $G$ do we have that $J(G \cup W(S))^{(s)}$ is componentwise linear for all $s \in \NN$? A number of special configurations of such subsets $S$ have been identified.
The case $s = 1$ in the following result of Selvaraja \cite{S2020} was already known by Francisco and H\`a in \cite{FH2008}.

\begin{theorem}[{\cite[Corollary 4.5]{S2020}}]
	\label{whiskering}
  Let $G$ be graph and let $S \subset V_G$ be a vertex cover of $G$. Then, $J(G \cup W(S))^{(s)}$ has linear quotients for any $s \ge 1$.
\end{theorem}

As a consequence to Theorem \ref{whiskering}, Selvaraja and Skelton \cite{SS2021} obtained the following corollary; the case $s=1$ was again known in \cite{FH2008}.

\begin{corollary}[{\cite[Corollary 4.6]{SS2021}}]
  Let $G$ be a graph and $S \subseteq G$.  If $|S| \geq |V(G)|-3$,
  then $J(G \cup W(S))^{(s)}$ is componentwise linear for all
  $s \geq 1$.
  \end{corollary}

The condition that $S$ is a vertex cover in Theorem \ref{whiskering} is improved by Gu, H\`a, and Skelton \cite{GHS2021}. We call a subset $S$ of the vertices in $G$ a \emph{cycle cover} if every cycle in $G$ contains at least a vertex in $S$. A vertex cover is necessarily a cycle cover, but contains a lot more vertices in general.

\begin{theorem}[{\cite[Theorem 3.10]{GHS2021}}]	\label{thm.GHS}
  Let $G$ be a graph and let $S$ be a cycle cover of $G$. Let
  $H$ be the graph obtained by adding at least one whisker
  to each vertex in $S$.  Then $J(H)^{(s)}$ is componentwise linear for
  all $s \geq 1$.
  \end{theorem}

Theorem \ref{thm.GHS} is slightly generalized further in \cite[Theorem 4.6]{GHS2021},
where it is shown that instead of adding just a whisker at each vertex
in $S$, one can add a {\it non-pure star complete graph}, a graph
constructed by adjoining complete graphs of different sizes
at a single vertex, which has at least one whisker (see \cite{GHS2021} for more details).

We round out this section by pointing out that there has been
little work into the case of symbolic powers of edge ideals.
We complete this paper with one result in
this direction that makes use of Theorem \ref{veronese}.

\begin{theorem}\label{newresult}
  Let $G$ be a graph and suppose that $W_1,\ldots,W_t$ is a complete
  list of minimal vertex covers.  Suppose that $W_i \cup W_j = V$
  for all $i \neq j$.  Then $I(G)^{(s)}$ is componentwise linear
  for all $s \geq 1$.
\end{theorem}

\begin{proof}  Suppose that $W_1,\ldots,W_t$ is a complete
  list of minimal vertex covers of $G$.  The edge ideal $I(G)$
  then has the following primary decomposition
  $$I(G) = \langle x ~|~ x \in W_1 \rangle \cap \cdots \cap \langle x ~|~
  x \in W_t \rangle.$$
  For a proof, see \cite[Corollary 1.35]{VT2013}.
  For a variable $x_i$, let ${\rm supp}(x_i) = \{i\}$.  Note
  we can rewrite $\langle x ~|~ x \in W_i \rangle  =\mathfrak{m}_{J_i}$
  as a Veronese ideal with
  $J_i = \{{\rm supp}(x) ~|~ x \in W_i\} \subseteq [n]$. Thus,
  since $I(G)$ is a square-free monomial ideal, we can write
  the $s$-th symbolic powers of $I(G)$ as
  $$I(G)^{(s)} = \mathfrak{m}_{J_1}^s \cap \cdots \cap \mathfrak{m}_{J_t}^{s}.$$
  The conclusion now follows by Theorem \ref{veronese}.
\end{proof}

A graph $G$ is a {\it complete $m$-partite graph} if the
vertices $V$ can be partitioned as $V = V_1 \cup V_2 \cup \cdots \cup V_m$
such that for all $i \neq j$, if $x \in V_i$ and $y \in V_j$, then
$\{x,y\} \in E$.  Note that the complete graph $K_n$ is
the complete $n$-partite graph with $V_i = \{x_i\}$ for all $i$.
We then have the following result.

\begin{corollary}\label{newresult2}
  Let $G$ be a complete $m$-partite graph (for any
  $m \geq 2$).  Then $I(G)^{(s)}$ is componentwise linear for
  any $s \geq 1$.
\end{corollary}

\begin{proof}
  If $V = V_1 \cup V_2 \cup \cdots \cup V_m$ is the partition
  of $V$, then the minimal vertex covers of $G$ have the
  form $W_i = V_1 \cup \cdots \cup \widehat{V}_i \cup \cdots \cup V_m$
  for $i=1,\ldots,m$, where we mean $V_i$ is omitted.  Now
  apply Theorem \ref{newresult}.
  \end{proof}


\section{Future research directions}

We finish this paper with some problems which we hope will
generate future work.

A natural way to generalize Herzog, Hibi, and Ohsugi's conjecture
is consider objects more general than chordal graphs.  Over the
last decade, there has been interest in generalizing the
property of chordality of graphs to simplicial complexes (see,
for example, \cite{ANS2016} and references therein).
The notion of a cover ideal can be generalized to simplicial complexes
as follows:  given a simplicial complex $\Delta$, the cover ideal
of $\Delta$ is
$$J(\Delta) = \langle x_{i_1}\cdots x_{i_s} ~|~ \{x_{i_1},\ldots,x_{i_s}\} \cap
F \neq \emptyset ~~\mbox{for all facets $F \in \Delta$}\rangle.$$
The following question then generalizes the Herzog-Hibi-Ohsugi
Conjecture:

\begin{question}
  Let $\Delta$ be a ``chordal'' simplicial complex (using an appropriate
  definition of chordal) with cover ideal $J(\Delta)$.  Is $J(\Delta)$
  componentwise linear?  Is $J(\Delta)^s$ componentwise linear for
  all $s \geq 1?$  We can also ask similar questions for
  $J(\Delta)^{(s)}$.
  \end{question}

\noindent
We have put chordal in quotes since it is not clear which
generalization of chordality one will want to use. A
starting point to attack this question would be the
case of simplicial trees, as defined by Faridi \cite{F2004}.  In
fact, when $\Delta$ is a simplicial tree, Faridi has already
shown that $J(\Delta)$ is componentwise linear
(see \cite[Corollary 5.5]{F2004}).

Theorems \ref{chordallinear} and \ref{bipartitelinear} show
that if $G$ is a chordal or bipartite graph such that $J(G)$ has a
linear resolution, then all powers of $J(G)$ have a linear resolution.
This leads to the following question, which was first posed by
Mohammadi \cite[Question 4.1]{M2014}.

\begin{question}\label{powersoflinear}
  Let $G$ be a graph such that the cover ideal $J(G)$ has a linear
  resolution.  Does $J(G)^s$ have a linear resolution for all $s \geq 1$?
  If not, what hypotheses are needed on $G$ to give this conclusion?
\end{question}

\noindent
Mohammadi has shown that the previous question is true for {\it cactus graphs}
(see \cite[Theorem 4.3]{M2014}).  A cactus graph is one where each edge belongs
to at most one induced cycle in the graph.
We are not aware of any other families
of graphs for which there is a positive (or negative!) answer to Question
\ref{powersoflinear}.

As noted in Section 4, we do not know of many examples
for edge ideals whose powers are componentwise linear.
We formalize this as a question.

\begin{question}\label{q.63}
  Let $G$ be a graph with edge ideal $I(G)$.  What properties
  on $G$ imply that $I(G)^s$ is componentwise linear?
  Similarly, what conditions imply $I(G)^{(s)}$ is componentwise linear?
  \end{question}
\noindent
Observe that $I(G)^s$ is generated in a single degree.
  Thus, for $I(G)^s$ to be componentwise linear, it
  needs to have a linear resolution. Hence, Question \ref{q.63} for
  $I(G)^s$  reduces to asking when a power of an edge ideal has a linear
  resolution.  See Peeva and Nevo \cite{PN2013} for
  some work in this direction.

We also add a question
from Selvaraja and Skelton's work (see \cite[Question 5.8]{SS2021}).

\begin{question} Let $G$ be a vertex decomposable graph.
  \begin{enumerate}
  \item If $J(G)^{(2)}$ is not componentwise linear, is it true
    that $J(G)^{(s)}$ is not componentwise linear for all $s \geq 3$?
  \item If $\mathcal{B}_{G \setminus N[A]}$ is vertex decomposable
    for any independent set $A$, is $J(G)^{(s)}$ a componentwise
    linear ideal for all $s \geq 2$?
    \end{enumerate}
\end{question}

We consider another question inspired by Selvaraja and Skelton's work,
namely, Theorem \ref{thm.nonCWLSS}.  As was shown in this
theorem, one can determine if $J(G)^{(s)}$ is not componentwise linear
by checking for small values of $s$.  This leads to a much
more general question:

\begin{question}
  Suppose that $I$ is an ideal that is
  componentwise linear.  Does there exists an integer $t \geq 1$
  such that if $I^{(i)}$ is componentwise linear
  for all $1 \leq i \leq t$, then  $I^{(s)}$
  is componentwise linear for all $s \geq t$?
  \end{question}

\noindent
An answer to the above question even in the case that
$I = I(G)$ or $J(G)$ would be of great interest, especially if the value
of $t$ is related to a graph invariant.  Note that the analogous question
  for regular powers has a negative answer.  In particular,
  Conca \cite[Theorem 3.1]{C2006} showed that for any integer $d>1$,
  there exists an ideal $I(d)$ such that $I(d)^k$ has a linear
  resolution (and hence, is componentwise linear) for all $1 \leq k < d$,
  but $I(d)^d$ does not have a linear resolution (and hence, is
  not componentwise linear).

As seen in Theorems \ref{thm.DHNT}, \ref{thm.Selva}, and \ref{thm.GHS},
the operation of whiskering can turn a graph $G$ into a new
graph $H$ such that $J(H)^{(s)}$
is componentwise linear for all $s \geq 1$.  We know of no similar
results for the regular powers.   We thus pose the following
question:

\begin{question}
  Given a graph $G$, can we attach whiskers to $G$ so that the resulting
  graph $H$ has the property that $J(H)^s$ is componentwise linear for
  all $s \geq 1$.  Can we classify all the ways to add whiskers to $G$
  to make $H$ so that $J(H)^s$, respectively $J(H)^{(s)}$, is componentwise linear.
  \end{question}

Moving beyond cover ideals, in \cite{BVT2013, BFHVT2015}, the authors
together with Biermann and Francisco gave a generalization for the
whiskering process that works also for simplicial complexes. We would like to understand if this process would produce more general monomial ideals with the property that all the symbolic powers of the Alexander dual of its Stanley-Reisner ideal are componentwise linear.

For a subset $W$ of the vertices of a simplicial complex $\Delta$, the \emph{restriction} of $\Delta$ on $W$, denoted by $\Delta|_W$, is the simplicial complex whose faces are $\{F \in \Delta ~\big|~ F \subseteq W\}$.

\begin{definition}
	Let $\Delta$ be a simplicial complex on the vertices $V$, let $W \subseteq V$, and let $t \in \NN$.
	\begin{enumerate}
	\item A \emph{partial $t$-coloring} of $\Delta|_W$ is
          given by a partition $W = W_1 \cup \dots \cup W_t$
          (where the set $W_i$ can be empty) such that no facet
          of $\Delta|_W$ contains more than one vertex in each $W_i$.

	\item Let $\chi$ be a partial $t$-coloring of $\Delta|_W$
          given by $W = W_1 \cup \dots \cup W_t$. We define a new
          simplicial complex $\Delta_\chi$ on the vertex set
	  $V \cup \{y_1, \dots, y_t\},$
	  with faces
	  $$\{\sigma \cup \tau ~\big|~ \sigma \in \Delta, \tau \subseteq \{y_1, \dots, y_t\} ~\mbox{and}~ \sigma \cap W_j = \emptyset \text{ if } y_j \in \tau\}.$$
	\end{enumerate}
\end{definition}

\begin{question}
	Let $\Delta$ be a simplicial complex and let $W$ be a subset of the vertices in $\Delta$. For a partial $t$-coloring $\chi$ of $\Delta|_W$, let $I_{\Delta_\chi}$ be the Stanley-Reisner ideal of $\Delta_\chi$ and let $J_{\Delta_\chi}$ be its Alexander dual. What conditions on $W$ and $\chi$ would guarantee that $J_{\Delta_\chi}^{(s)}$ is componentwise linear for all $s \ge 1$?
\end{question}


\end{document}